\documentclass{article}
\usepackage{amsthm, amsfonts, amssymb,latexsym}

\title{Improved Bounds for Pencils of Lines}

\newcommand{\Z}{\mathbb{Z}}
\newcommand{\R}{\mathbb{R}}
\usepackage{tikz}
\usepackage{amsmath}
\usepackage[a4paper, total={6in, 8in}]{geometry}
\newcommand{\cL}{\mathcal{L}}

\author{Oliver Roche-Newton and Audie Warren}
\newtheorem{prop}{Proposition}
\newtheorem{lem}{Lemma}

\newtheorem{thm}{Theorem}
\newtheorem{cor}{Corollary}
\newtheorem{pro}{Problem}

\begin{document}
 \maketitle
 
\begin{abstract}
We consider a question raised by Rudnev: given four pencils of $n$ concurrent lines in $\mathbb R^2$, with the four centres of the pencils non-collinear, what is the maximum possible size of the set of points where four lines meet? Our main result states that the number of such points is $O(n^{11/6})$, improving a result of Chang and Solymosi \cite{CS}.

  We also consider constructions for this problem.  Alon, Ruzsa and Solymosi \cite{Alon} constructed an arrangement of four non-collinear $n$-pencils which determine $\Omega(n^{3/2})$ four-rich points. We give a construction to show that this is not tight,  improving this lower bound by a logarithmic factor. We also give a construction of a set of $m$ $n$-pencils, whose centres are in general position, that determine $\Omega_m(n^{3/2})$ $m$-rich points.
\end{abstract}

 \section{Introduction}

An $n$-pencil with centre $p \in P^2(\R)$ is defined to be a set of $n$ concurrent lines passing through $p$. Given $m$ $n$-pencils, a point is said to be $m$-rich if one line from each of the pencils passes through it. The question we study in this paper is the following: what is the maximum possible size of the set of $m$-rich points determined by $m$ $n$-pencils?

The first interesting case is when $m=4$. For $m=2,3$ there are natural constructions giving $\Omega(n^2)$ $m$-rich points, which is certainly maximal.\footnote{For $m=2$, any two $n$-pencils with distinct directions determine exactly $n^2$ crossing points. For $m=3$, one can take two of the centres of the pencils on the line at infinity so that their crossing points give a grid $A \times A$ where $A$ is a geometric progression. Choosing the origin as the centre for the third pencil, $\Omega(n^2)$ of the points of $A \times A$ can be covered by $n$ lines through the origin by using the ratio set as the set of slopes.} Furthermore, when $m=4$ and the centres of the four pencils are collinear, it is still possible\footnote{One way to see this is by taking the four centre points on the line at infinity. The first two pencils again intersect in a grid $A \times A$, and this time we make $A=\{1,2,\dots,n\}$. The second two pencils give a family of lines with slopes $1$ and $-1$ respectively, and both directions give rise to a family of lines of size $2n-1$ which cover $A \times A$. Thus we have four pencils of size $O(n)$ (with their centres collinear) and $n^2$ $4$-rich points.} to give a construction generating $\Omega(n^2)$ $4$-rich points. With these degenerate cases dismissed, we arrive at the following two questions of Rudnev.

\begin{pro} \label{problem1} Given four $n$-pencils whose centres do not lie on a single line, what is the maximum possible size of the set of $4$-rich points they determine?
\end{pro}

\begin{pro} \label{problem2} Given four $n$-pencils whose centres are in general position (i.e. no three of the centres are collinear), what is the maximum possible size of the set of $4$-rich points they determine?
\end{pro}

It is possible that the answers to these two questions are the same.

Some progress on the first problem was given in a recent paper of Alon, Ruzsa and Solymosi \cite{Alon}. They gave a construction of four $n$-pencils with non-collinear centres which determine $\Omega(n^{3/2})$ $4$-rich points. From the other side, a result of Chang and Solymosi \cite{CS} implies that for any four $n$-pencils with non-collinear centres, the number of $4$-rich points is $O(n^{2-\delta})$. Their proof gives the value $\delta=1/24$.

The main results of this paper are the following two theorems, which give improved upper and lower bounds  respectively for the maximum possible number of $4$-rich points.

\begin{thm}\label{upper}
Let $P$ be the set of $4$-rich points defined by a set of four non-collinear $n$-pencils. Then we have 
$$|P| =O( n^{11/6}).$$
\end{thm}

\begin{thm}\label{lower}
There exist four $n$-pencils with non-collinear centres which determine $\Omega(n^{3/2} \log^c n)$ $4$-rich points, for some absolute constant $c>0$.
\end{thm}

The construction given in \cite{Alon} of four pencils determining $\Omega(n^{3/2})$ had three of the centres on a line, and thus it did not immediately give any progress towards Problem \ref{problem2}. We give a similar construction with no three of the centres on a line.

\begin{thm}\label{lowergen}
There exist four $n$-pencils, whose centres are in general position, which determine $\Omega(n^{3/2} )$ $4$-rich points.
\end{thm}

Furthermore, we generalise this to give a construction of $m$ $n$-pencils determining many $m$-rich points.

\begin{thm}\label{lowergenm}
For any $m \in \mathbb{N}$, there exist $m$ $n$-pencils whose centres are in general position which determine $\Omega_m(n^{3/2} )$ $m$-rich points.
\end{thm}

For a precise version of this result with the dependence on $m$ made explicit, see the forthcoming Proposition \ref{prop:lowergenm}.

\subsection{Notation}  Throughout this paper, the standard notation
$\ll,\gg$ and $O, \Omega$  is applied to positive quantities in the usual way. $X\gg Y$, $Y \ll X$, $X=\Omega(Y)$ and $Y=O(X)$ all mean that $X\geq cY$, for some absolute constant $c>0$.  

\section{Connection with the sum-product problem}

The construction relating to Problem \ref{problem1} given in \cite{Alon} arose from some surprising constructions for the sum-product problem restricted to graphs.  For a finite set $A \subseteq \R$, define the sum and product set as 
 $$A+A = \{ a+b : a,b \in A\}$$
 $$A \cdot A = \{ ab : a,b \in A\}.$$
 We can also define the difference and ratio set in an analogous way. The famous Erd\H{o}s - Szemer\'{e}di conjecture states that for all $\epsilon >0$, there exists an absolute constant $c(\epsilon)$ such that for all finite $A \subset \mathbb Z$
 $$\max \{ |A+A|, |AA| \} \geq c(\epsilon) |A|^{2-\epsilon}.$$
Erd\H{o}s and Szemer\'{e}di also considered taking sums and products restricted to a specified subset of $A \times A$, as follows. Let $G$ be a bipartite graph with vertices being two distinct copies of $A$, and let $E(G) \subseteq A \times A$ be the edges of $G$. We define the sumset of $A$ along $G$ to be 
 $$A +_G A = \{ a + b : (a,b) \in E(G) \}.$$
In more generality, for $A$ and $B$ two finite subsets of $\R$, we take a set of edges $E(G) \subseteq A \times B$, and define the sum set
$$A +_G B = \{ a + b : (a,b) \in E(G) \}.$$
The restricted product set, ratio set etc. are defined in the same way. Erd\H{o}s and Szemer\'{e}di also gave a stronger version of their conjecture in this restricted setting, essentially saying that for sufficiently dense graphs $G \subset A \times A$, at least one of $|A+_G A|$ or $|A \cdot_G A|$ is close to $|G|$. In \cite{Alon}, the authors gave several constructions to show that this stronger conjecture, and variants thereof, do not hold. One such result was the following.


\begin{thm}[Alon, Ruzsa, Solymosi]
For arbitrarily large $n$, there exists $A \subseteq \R$ finite with $|A| = \Theta(n)$, and a subset $S \subseteq A \times A$ with $|S| = \Omega(n^{3/2})$, such that $S$ is the set of edges of a graph $G$ with
$$|A+_GA| + |A /_G A| = O(n).$$
\end{thm}

Both the sumset and the ratio set are at most linear in size, but the graph has many edges. The construction used in this theorem is then converted, via a projective transformation, into a construction of a set of four $n$-pencils of lines, with non-collinear centres, that determine $\Omega(n^{3/2})$ 4-rich points. 

Similarly, our results in Theorems \ref{upper}, \ref{lower} and \ref{lowergen} follow from considering sum-product type problems restricted to graphs. The sum-product problem that is most relevant to this paper is that of showing that if the product set of $A$ is small, then the product set of a shift of $A$ must be large. In this direction, it was proven by Garaev and Shen \cite{GS}, that for any finite $A,B,C \in \mathbb R$ and any non-zero $x \in \mathbb R$,
\begin{equation}
 |AB|, |(A+x)C| \gg |A|^{3/4}|B|^{1/4}|C|^{1/4} .
 \label{GSbound}
 \end{equation}
 This result and its proof closely follow the seminal work of Elekes \cite{E} in which the Szemer\'{e}di-Trotter Theorem was first used to prove sum-product results. 
 
 In the process of proving Theorems \ref{upper}, \ref{lower}, and \ref{lowergen}, we obtain some results about this version of the sum-product problem restricted to graphs which may be of independent interest. For example, we prove the following result.


\begin{thm} \label{lowerSP}
For arbitrarily large $n$, there exists $A,B \subseteq \mathbb{Q}$ with $|A|, |B| \gg n$, and a subset $S \subseteq A \times B$ with $|S| = \Omega(n^{3/2}\log(n)^{\frac{43}{1000}})$, such that $S$ is the set of edges of a graph $G$ with
$$|A /_G B| + |(A+1) /_G B| + |(A+2) /_G B| \ll n.$$
\end{thm}
In the above $A /_G B :=\{a/b : (a,b) \in E(G) \}$. More generally, for any $x, y \in \mathbb R$,
\[(A+x) /_G (B+y) := \left \{\frac{a+x}{b+y} : (a,b) \in E(G)\right \}.
\]

Finally, since we will use the Szemer\'{e}di-Trotter Theorem in the forthcoming section, we state it below.

\begin{thm}[Szemer\'{e}di-Trotter Theorem] Let $P \subset \mathbb R^2$ be finite and let $L$ be a finite set of lines in $\mathbb R^2$. Then
\[I(P,L):=|\{(p,l) \in P \times L : p \in l \}| \ll (|P||L|)^{2/3} +|P|+|L|.
\]
\end{thm}

 
 \section{Proof of Theorem \ref{upper}}

 We begin by giving a way to translate a question concerning pencils into a question concerning ratio and sum sets. The setup here is similar to that of Chang and Solymosi \cite{CS}. 
 
 We take four non-collinear pencils $\cL_1$, $\cL_2$, $\cL_3$, and $\cL_4$, with $|\cL_i| = n$ for each $i$. As they are non-collinear, there exists a pair (say $\cL_1$ and $\cL_2$) such that the line connecting the centres of these pencils does not contain the centre of $\cL_3$ or $\cL_4$. We apply a projective transformation to send the centres of $\cL_1$ and $\cL_2$ to the projective coordinates $(1;0;0)$ and $(0 ;1;0)$ respectively. $\cL_1$ now consists of horizontal lines, and $\cL_2$ of vertical lines. By the choice we made, both the pencils $\cL_3$ and $\cL_4$ have affine centres. 
 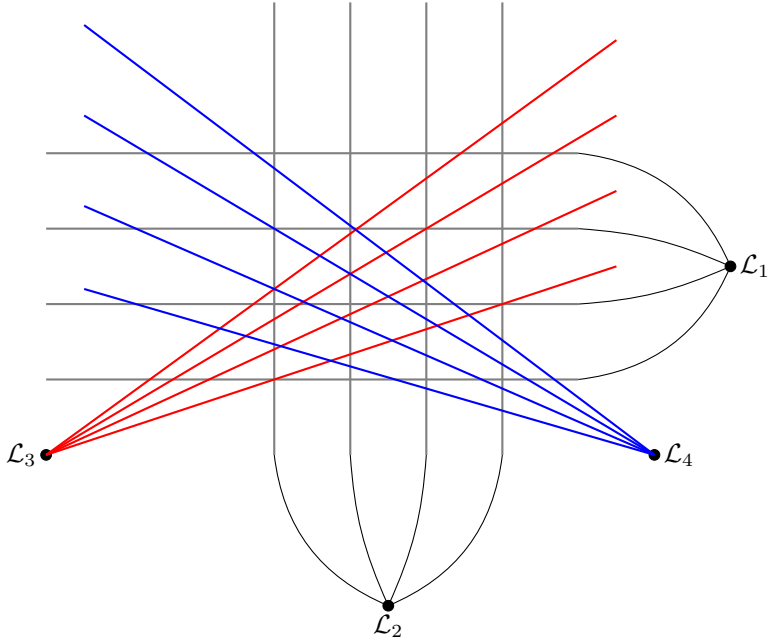
\begin{figure}[h!]
     \centering
     \begin{tikzpicture}
\draw[gray, thick] (-4,-1) -- (3,-1);
\draw[gray, thick] (-4,0) -- (3,0);
\draw[gray, thick] (-4,1) -- (3,1);
\draw[gray, thick] (-4,2) -- (3,2);
\draw[gray, thick] (-1,-2) -- (-1,4);
\draw[gray, thick] (0,-2) -- (0,4);
\draw[gray, thick] (1,-2) -- (1,4);
\draw[gray, thick] (2,-2) -- (2,4);

\filldraw[black] (5,0.5) circle (2pt) node[anchor=west] {$\cL_1$};
\draw[bend left] (5,0.5) to (3,-1);
\draw[bend left = 10] (5,0.5) to (3,0);
\draw[bend right = 10] (5,0.5) to (3,1);
\draw[bend right] (5,0.5) to (3,2);

\filldraw[black] (0.5,-4) circle (2pt) node[anchor=north] {$\cL_2$};
\draw[bend left] (0.5,-4) to (-1,-2);
\draw[bend left = 10] (0.5,-4) to (0,-2);
\draw[bend right = 10] (0.5,-4) to (1,-2);
\draw[bend right] (0.5,-4) to (2,-2);

 \filldraw[black] (-4,-2) circle (2pt) node[anchor=east] {$\cL_3$};
 \draw[red, thick] (-4,-2) -- (3.5,3.5);
  \draw[red, thick] (-4,-2) -- (3.5,2.5);
   \draw[red, thick] (-4,-2) -- (3.5,1.5);
    \draw[red, thick] (-4,-2) -- (3.5,0.5);
    
     \filldraw[black] (4,-2) circle (2pt) node[anchor=west] {$\cL_4$};
 \draw[blue, thick] (4,-2) -- (-3.5,2.5);
  \draw[blue, thick] (4,-2) -- (-3.5,3.7);
   \draw[blue, thick] (4,-2) -- (-3.5,1.3);
    \draw[blue, thick] (4,-2) -- (-3.5,0.2);
\end{tikzpicture}
     \caption{An example of four pencils after a projective transformation.}
     \label{fig:my_label}
 \end{figure}

Pencils $\cL_1$ and $\cL_2$ define a cartesian product $A \times B$, where $|A|,|B| = n$. Let $S\subseteq A \times B$ be the set of $4$-rich points. Let $(x_1,y_1)$ and $(x_2,y_2)$ be the centres of $\cL_3$ and $\cL_4$ respectively. Both $\cL_3$ and $\cL_4$ cover $S$, and by identifying an element $\lambda$ of $(A-x_1) /_G (B-y_1)$ with its corresponding line of slope $\lambda$ through $(x_1,y_1)$, we have
\begin{equation} \label{size}(A-x_1) /_G (B-y_1) \subseteq \cL_3 \implies | (A-x_1) /_G (B-y_1) | \leq n \quad \end{equation}
$$(A-x_2) /_G (B-y_2) \subseteq \cL_4 \implies | (A-x_2) /_G (B-y_2) | \leq n,$$
where $G$ is the bipartite graph on $A \times B$ induced by taking the set of edges to be $S$. We see that the question now concerns bounding $S$, the amount of edges of the graph $G$. We prove the following lemma, which is based on the proof of inequality \eqref{GSbound} given in \cite{GS}.

\begin{lem} \label{ratio}
Let $A$, $B$ be finite sets of real numbers, and let $|A| = |B| = n$. Let $(x_1,y_1)$, $(x_2,y_2)$ be two distinct points in $\R^2$, and let $G$ be a bipartite graph on $A \times B$ . Then
$$|(A - x_1) /_G (B-y_1)| + |(A - x_2) /_G (B-y_2)| \gg \frac{|E(G)|^{3/2}}{n^{7/4}}.$$ 
\end{lem}

\begin{proof}
Since the points $(x_1,y_1)$ and $(x_2,y_2)$ are distinct, at least one of $x_1 \neq x_2$ or $y_1 \neq y_2$ holds. We will assume without loss of generality that $x_1 \neq x_2$. We also assume, without loss of generality, that $y_1,y_2 \notin B$, so as to avoid issues with division by zero.

Furthermore, we can assume that $|E(G)| \geq Cn^{3/2}$ for some sufficiently large constant $C$, as otherwise the result holds for trivial reasons.
Indeed, for any $x_1 \in \mathbb R$, $y_1 \in \mathbb R \setminus B$ and any graph $G$ on $A \times B$ with $|E(G)| \ll n^{3/2}$,
\[|(A-x_1)/_G(B-y_1)| \geq \frac{|E(G)|}{|A|} \gg \frac{|E(G)|^{3/2}}{|A|^{7/4}}.
\]

Let $P =(A - x_1) /_G (B-y_1) \times (A - x_2) /_G (B-y_2)$. Define the line $l_{b_1,b_2}$ by the equation $ (b_2 - y_2)y= (b_1 - y_1)x + (x_1 - x_2)  $, and let $L = \{ l_{b_1,b_2} : b_1,b_2 \in B\}$. Since, $x_1 \neq x_2$, all of these lines are distinct, and so $|L| = |B|^2 = n^2$. For each $a \in A$, if $(a,b_1), (a,b_2) \in E(G)$, the pair $\left( \frac{a-x_1}{b_1-y_1}, \frac{a-x_2}{b_2-y_2} \right) \in P$ lies on line $l_{b_1,b_2}$. For $a \in A$, let $N(a)$ denote the neighbourhood of $A$ in $G$, that is, $N(a):=\{b \in B : (a,b) \in E(G) \}$. Then we have a bound for the number of incidences:
\begin{align*}
    I(P,L) &\geq \sum_{a\in A}|N(a)|^2\\
    & \geq \frac{|E(G)|^2}{n}
\end{align*}
by Cauchy-Schwarz. We use the Szemer\'{e}di-Trotter Theorem to bound on the other side as
\[ 
\frac{|E(G)|^2}{n} \ll |P| + |L| + (|P||L|)^{2/3} .
\]
Since $|E(G)| \geq Cn^{3/2}$ and $|L|=n^2$, the middle term here can be dismissed and we have
\begin{equation}
\frac{|E(G)|^2}{n} \ll |P|  + (|P||L|)^{2/3} .
\label{nearly}
\end{equation}

If the second term on the right-hand side dominates, we get
$$\Big[|(A - x_1) /_G (B-y_1)||(A - x_2) /_G (B-y_2)|\Big]^{2/3}n^{4/3} \gg \frac{|E(G)|^2}{n} ,$$
and so
$$|(A - x_1) /_G (B-y_1)| + |(A - x_2) /_G (B-y_2)| \gg \frac{|E(G)|^{3/2}}{n^{7/4}}.$$
If, on the other hand, the first term on the right hand side of \eqref{nearly} dominates, we get a stronger inequality than that claimed in the statement of the lemma, and so the proof of Lemma \ref{ratio} is complete.
\end{proof}



Continuing with our four pencils from before, we had the information from the inequalities (\ref{size}), which when we combine with Lemma \ref{ratio} gives 
$$n \gg |(A - x_1) /_G (B-y_1)| + |(A - x_2) /_G (B-y_2)| \gg \frac{|E(G)|^{3/2}}{n^{7/4}}$$
so that the number of edges, and thus the number of four-rich points, satisfies
$$|E(G)| \ll n^{11/6}.$$
This concludes the proof of Theorem \ref{upper}. \qed

This argument can be repeated to give similar results in other fields by using a suitable replacement for the Szemer\'{e}di-Trotter Theorem. In the complex setting we can use a result of Toth \cite{T} (see also Zahl \cite{Z}), obtaining the same results as above. Over $\mathbb F_p$ we can use an incidence theorem for cartesian products due to Stevens and de Zeeuw \cite{SdZ}. We calculated that this gives an upper bound $O(n^{2-\frac{1}{8}})$ for the number of $4$-rich points.

\section{Proof of Theorem \ref{lower}}
In order to prove Theorem \ref{lower}, we will first prove Theorem \ref{lowerSP}. We will then show this sum-product construction implies a construction with four pencils determining many $4$-rich points.

We make use of the following theorem due to Ford \cite{ford} concerning the product set of the first $n$ integers.
\begin{thm}
Let $A(n)$ be the number of positive integers $m \leq n$ which can be written as a product $m = m_1m_2$, where $m_1$, $m_2 \in \{1,2,...,\lfloor \sqrt{n} \rfloor \}$. Then
$$A(n) \sim \frac{n}{(\log n)^\delta (\log \log n)^{3/2}}$$
where $\delta = 1 - \frac{1 + \log \log 2}{\log 2} = 0.086071\dots$.
\end{thm}
As a corollary, we re-write this theorem in the language of product sets.
\begin{cor}\label{productset}
Let $A = \{1,2,...,n \}$. Then the product set $AA$ has size 
$$|AA| \ll \frac{n^2}{(\log n)^{\frac{43}{500}}}.$$
\end{cor}
Here we have absorbed the $\log \log $ factor by slightly reducing the exponent of the log factor, for simplicity of the forthcoming calculations. We now have the tools to prove Theorem \ref{lowerSP}.
\begin{proof}[Proof of Theorem \ref{lowerSP}]
Let $d>0$ be some parameter to be chosen later. Define the sets
\begin{equation}
    A = \left\{ \frac{i}{j} : i,j \in \Z, \ (i,j)=1, \  1 \leq i,j \leq \frac{\sqrt{n}}{(\log n)^d}, \ j \geq \frac{\sqrt{n}}{2(\log n)^d} \right\}
\end{equation}
\begin{equation}
    B = \left\{ \frac{1}{l} : l \in \Z, \ 1 \leq l \leq \frac{n}{(\log n)^d} \right\}.
\end{equation}
Note that we have the size of $A$ being
$$|A| \sim \frac{n}{(\log n)^{2d}}.$$
Indeed, the number of coprime pairs of integers less than some parameter $x$ is asymptotically equal to $\frac{6}{\pi^2}x^2$, and so
\[ |A| \geq \frac{6}{\pi^2} \left( \frac{\sqrt{n}}{(\log n)^d} \right )^2 -  \frac{6}{\pi^2} \left( \frac{\sqrt{n}}{(2\log n)^d} \right )^2 + \text{ lower order terms } \gg \frac{n}{(\log n)^{2d}}.
\]

We define a bipartite graph on $A \times B$, where the edges $E(G)$  are defined by the following.
$$E(G) = \left\{ \left(\frac{i}{j}, \frac{1}{l} \right) \in A \times B : j | l \right\}.$$
The number of edges is given by the formula
$$|E(G)| = \sum_j \Big| \Big\{ i : (i,j)=1 \Big\} \Big| \ \left|\left\{k \in \Z : 1 \leq kj \leq \frac{n}{(\log n)^d} \right\} \right| .$$
The size of the set $\left\{k \in \Z : 1 \leq kj \leq \frac{n}{(\log n)^d} \right\}$ gives the amount of multiples of $j$ up to $\frac{n}{(\log n)^d}$. As $j \leq \frac{\sqrt{n}}{(\log n)^d}$, a lower bound for the amount of these multiples is $\sqrt{n}$. We can thus move this outside of the sum over $j$, obtaining
$$|E(G)| \geq \sqrt{n} \sum_j \Big| \Big\{ i : (i,j)=1 \Big\} \Big| = \sqrt n |A| \gg \frac{n^{3/2}}{(\log n)^{2d}} . $$
The ratio set $A /_G B$ consists of the elements
\begin{align*}
A /_G B &= \left\{ \frac{il}{j} \text{ such that } \frac{i}{j} \in A, \frac{1}{l} \in B, \ j|l \right\} \\
& \subseteq \left\{ i l^{\prime} : 1 \leq i \leq \frac{\sqrt{n}}{(\log n)^d}, \ 1 \leq l^{\prime} \leq 2\sqrt{n} \right\}\\
& \subseteq CC
\end{align*}
where $C = \{1,2,...,2 \sqrt{n}\}$. Thus we have\footnote{It is possible to be more careful here, and use an analogue of Ford's result for an asymmetric multiplication table, in order to make a saving in the exponent of the logarithmic factor in Theorem \ref{lowerSP} and thus in turn Theorem \ref{lower}. In order to simplify the calculations we do not pursue this improvement.} by Corollary \ref{productset}
$$|A /_G B| \ll \frac{n}{(\log n)^{\frac{43}{500}}}.$$

When we apply a shift of $1$ to $A$ and calculate the ratio set $(A+1) /_G B$, we get the same result.
\begin{align*}
(A+1) /_G B &= \left\{ \frac{(i+j)l}{j} : \frac{i}{j} \in A, \frac{1}{l} \in B, \ j|l \right\} \\
& \subseteq \left\{ (i+j) l^{\prime} : 1 \leq i \leq \frac{\sqrt{n}}{(\log n)^d},\ \frac{\sqrt{n}}{2(\log n)^d} \leq j \leq \frac{\sqrt{n}}{(\log n)^d} \ , 1 \leq l^{\prime} \leq 2\sqrt{n} \right\}\\
& \subseteq \left\{ k l^{\prime} : 1 \leq k \leq \frac{2\sqrt{n}}{(\log n)^d}, \ 1 \leq l^{\prime} \leq 2\sqrt{n} \right\}
\subseteq CC.
\end{align*}
For $(A+2) /_G B$ we find an extra constant, but we still have the same result. We now have the sum 
$$|A /_G B| + |(A+1) /_G B| + |(A+2) /_G B| \ll \frac{n}{(\log n)^{\frac{43}{500}}}$$
where the amount of edges on $G$ is 
$$|E(G)| \gg \frac{n^{3/2}}{(\log n)^{2d}}.$$
We now set $d = \frac{43}{1000}$, and let $m = \frac{n}{(\log n)^{\frac{43}{500}}}$. This gives us the following;
$$|B| \gg |A| \gg \frac{n}{(\log n)^{\frac{43}{500}}} = m$$
$$ |A /_G B| + |(A+1) /_G B| + |(A+2) /_G B| \ll \frac{n}{(\log n)^{\frac{43}{500}}} = m$$
$$|E(G)| \gg \frac{n^{3/2}}{(\log n)^{2d}} \gg m^{3/2}(\log m)^{\frac{43}{1000}},$$
thus completing the proof.
\end{proof}

We can immediately use this result to create a set of four pencils with many  $4$-rich points.


\begin{proof}[Proof of Theorem \ref{lower}]
We consider our construction from Theorem \ref{lowerSP}. The edges of the graph correspond to a set $S \subseteq A \times B \subset \R^2$. The amount of elements of $A /_G B$ and the two shifts are exactly the amount of lines needed to cover $S$ through either the origin for $A /_G B$, the point $(-1,0)$ for $(A+1) /_G B$ or $(-2,0)$ for $(A+2) /_G B$. These are our first three pencils, which we already know have cardinality $O(m)$. Our fourth pencil will have its centre on the line at infinity, and will consist of vertical lines covering $S$. The amount needed is precisely $|A|=O(m)$. The amount of $4$-rich points is at least the size of $S$, since each pencil covers $S$. Thus we have at least $m^{3/2} (\log m)^{\frac{43}{1000}}$  $4$-rich points.

Note also that the centres of the four pencils we have chosen are non-collinear. The point at infinity met by the line connecting $(0,0)$, $(-1,0)$ and $(-2,0)$ is not the equal to the point corresponding to the centre of the fourth pencil.
\end{proof}

\section{Constructions with arbitrarily many pencils}
 
 We give a construction of a set where the sum-set, ratio set, an additive shift of the ratio set, and the difference set are all linear when we restrict to a graph, where the graph has many edges. We also show using shifts of ratio sets that there are sets of $m$ $n$-pencils of lines that determine $\Omega_m(n^{3/2})$ $m$-rich points. 
 \begin{thm}
 For arbitrarily large $n$, there exists a set $A$ with $|A| = \Theta(n)$, and a graph $G$ on $A \times A$ with $\Omega(n^{3/2})$ edges, such that 
 $$|A +_G A| +  |A /_G A| + |(A+1) /_G (A+1)| + |A -_G A| \ll n.$$
 \end{thm}
 
 \begin{proof}
 Let 
 \[A := \left\{ \frac{i}{j} : (i,j)=1, \ 1 \leq i,j \leq  \sqrt n \right\}
 \] The size of $A$ is the amount of coprime pairs from $1$ to $\sqrt n$; therefore $|A| = \Theta(n)$. We define a bipartite graph $G$ with vertex set $A \times A$ and 
 \[E(G)= \left\{ \left(\frac{i}{j}, \frac{k}{j} \right )  : 1, \leq i,j,k \leq \sqrt n, (i,j)=1=(k,j)\right \}.
 \]
  With this definition, we have $|E(G)| \gg n^{3/2}$. Indeed, 
 \begin{align*}|E(G)|&= \sum_{1 \leq j \leq \sqrt n}| \{(i,k) : 1 \leq i,k \leq \sqrt n, (i,j)=1=(k,j) \}|
 \\ & =  \sum_{1 \leq j \leq \sqrt n} |\{i : 1 \leq i \leq \sqrt n, (i,j)=1 \}|^2,
\end{align*}
and so by the Cauchy-Schwarz inequality,
\begin{align*}n^2 &\ll \left( \sum_{1 \leq j \leq \sqrt n} |\{i : 1 \leq i \leq \sqrt n, (i,j)=1 \}|\right)^2 
\\& \leq \sqrt n  \sum_{1 \leq j \leq \sqrt n} |\{i : 1 \leq i \leq \sqrt n, (i,j)=1 \}|^2 = \sqrt n |E(G)|,
\end{align*}
as claimed.

 \begin{itemize}

 \item The sum set restricted to $G$ is $A+_GA \subseteq \left\{ \frac{i + k}{j} : i,j,k \in [\sqrt{n}] \right\} $. The numerator ranges from $1$ to $2\sqrt{n}$, and the denominator from $1$ to $\sqrt{n}$, thus $|A+_GA| \ll n$.
 
  \item The ratio set is $A /_G A \subseteq \left\{ \frac{i}{k} :i,k \in [\sqrt{n}] \right\}= A$, so $|A /_G A| \ll n$.
 
 \item The shifted ratio set is $(A+1) /_G (A+1) \subseteq \left\{ \frac{i+j}{k+j} :i,j,k \in [\sqrt{n}] \right\}$ and so $|(A+1) /_G (A+1)| \ll n$.
 
 \item Finally, the difference set is $A -_G A \subseteq \left\{ \frac{i-k}{j} : i,j,k \in [\sqrt{n}]\right\}$, so $|A -_G A| \ll n$.
  \end{itemize}
 Therefore the sum of the sizes of these four sets is $\ll n$.
 \end{proof}
Using the same construction, we may consider only ratio sets to generalise this to any number of pencils. We may arbitrarily shift the ratio set by any $(x,y) \in \Z^2$ and keep its size linear in $n$;
 \begin{align*}
 (A+x) /_G (A+y) & \subseteq \left\{ \frac{i+xj}{k+yj} : i,j,k \in [\sqrt{n}]\right\} \\
 \implies |(A+x) /_G (A+y)| &\leq (\sqrt{n} + x\sqrt{n})(\sqrt{n} + y\sqrt{n}) \ll xyn,
 \end{align*}
 which gives a construction to prove the following proposition, a more precise version of Theorem \ref{lowergenm}.
 
 \begin{prop} \label{prop:lowergenm}
 For any $m \in \mathbb{N}$, there exists a set of $m$ pencils of lines, with any three centres of pencils non-collinear, such that each pencil contains $ N$ lines, and the amount of $m$-rich points is $\Omega(N^{3/2}/m^3)$.
 \end{prop}
 
 \begin{proof}
To get the best possible dependence on $m$ in this statement, we need to choose a set of $m$ centres which are in general position, and so that their coordinates are as small as possible. It is possible to construct such a set of size $m$ in the lattice $[m] \times [m]$. We take $P$ to be this set of centres.

Let $A$ and $G$ be defined as above. Form $(A+x) /_G (A+y)$ for $(x,y) \in P$. The centres are non-collinear, each pencil contains $\ll m^2n:=N$ lines, and the amount of $m$-rich points is at least the amount of edges, thus $\Omega(n^{3/2})=\Omega(N^{3/2}/m^3)$.
 \end{proof}
 
 Finally, note that by taking $m=4$ in the previous proposition, we obtain Theorem \ref{lowergen}.
 
\subsection*{Acknowledgements}

Both authors were supported by the Austrian Science Fund (FWF) Project P 30405-N32. We thank Mehdi Makhul and Misha Rudnev for helpful conversations.


\begin{thebibliography}{9}


\bibitem{Alon} N. Alon, I. Ruzsa and J. Solymosi, `Sums, products and ratios along the edges of a graph', eprint arXiv:1802.06405, 18 Feb 2018.

\bibitem{CS} M.-C. Chang and J. Solymosi, `Sum-product theorems and incidence geometry',  \textit{J. Eur. Math. Soc.} 9 (2007), no.3, 545-560.

\bibitem{E} G. Elekes, `On the number of sums and products', \textit{Acta Arith.} 81 (1997), 365-367.

\bibitem{ford} K. Ford, `The distribution of integers with a divisor in a given interval',  \textit{Ann. of Math.} (2) 168 (2008), no. 2, 367-433. 

\bibitem{GS} M. Garaev and C.-Y. Shen, `On the size of the set $A(A+1)$', \textit{Math. Z.}, 265, no. 1, (2010), 125-132.

\bibitem{SdZ} S. Stevens and F. de Zeeuw, `An improved point-line incidence bound over arbitrary fields', \textit{Bull. Lond. Math. Soc.}, 49, no. 5, (2017), 842-858.

\bibitem{T} C. T\'{o}th, `The Szemer\'{e}di-Trotter theorem in the complex plane', \textit{Combinatorica}, 35, no. 1, (2015), 95-126.

\bibitem{Z} J. Zahl, `A Szemer\'{e}di-Trotter type theorem in $R^4$', \textit{Discrete Comput. Geom.}, 54, no. 3, (2015), 513-572.





 

\end{thebibliography}
\end{document}